\documentclass[graybox,envcountresetchap]{svmult}

\usepackage{amsfonts,amstext,amssymb,physics} 
\usepackage{amsmath}

\usepackage{mathptmx}       
\usepackage{bm}            
\usepackage{helvet}         
\usepackage{courier}        
\usepackage{type1cm}        
\usepackage{makeidx}         
\usepackage{graphicx}        
\usepackage{multicol}        
\usepackage[bottom]{footmisc}

\makeindex             

\usepackage[utf8]{inputenc}
\usepackage[T1]{fontenc}
\usepackage{mathtools}

\usepackage{amscd}

\usepackage{float}

\usepackage{longtable}

\usepackage{color}
\usepackage[all]{xy}

\usepackage{enumitem}  
\usepackage{calc}
\setlist{labelindent=1pt,itemsep=0.1cm}
\setlist[itemize]{leftmargin=1.2cm}
\setlist[enumerate]{itemindent=0em,leftmargin=1.2cm}
\setlist[enumerate,1]{label={\upshape(\roman*)}}

\usepackage[bookmarks=false,hidelinks,unicode]{hyperref}

\usepackage{cite}

\usepackage{marginnote}

\usepackage{tikz-cd}

\usepackage{nicematrix}

\allowdisplaybreaks[3]

\smartqed

\begin{document}

\title*{Generalization of Strang-Fix remark for discrete maximum principles on tetrahedral meshes}
\titlerunning{Discrete maximum principles on tetrahedral meshes}
\author{Safdar Ali \and Maksat Ashyraliyev \and  Sergey Korotov \and Sergei Silvestrov}
\authorrunning{S. Ali, M. Ashyraliyev,  S. Korotov, S. Silvestrov}
\institute{Safdar Ali \at Department of Business and Mathematics,
M\"alardalen University, Box 883, 721 23 V\"aster\aa s, Sweden, \email{safdar.ali@mdu.se}
\and Maksat Ashyraliyev \at Department of Business and Mathematics,
M\"alardalen University, Box 883, 721 23 V\"aster\aa s, Sweden, \email{maksat.ashyralyyev@mdu.se}
\and Sergey Korotov \at Department of Business and Mathematics,
M\"alardalen University, Box 883, 721 23 V\"aster\aa s, Sweden, \email{sergey.korotov@mdu.se}
\and
Sergei Silvestrov \at Department of Business and Mathematics,
M\"alardalen University, Box 883, 721 23 V\"aster\aa s, Sweden, \email{sergei.silvestrov@mdu.se}
}
%
%


\maketitle \label{chap:AliAshKorSil}

\vspace{-1cm}
\abstract{Most of qualitative properties of finite element (FE) approximations  for various PDE models, in particular the validity of the discrete maximum principles (DMPs), are usually guaranteed by certain geometric properties (e.g. acuteness/nonobtuseness) of FE meshes used.
In this work, we revisit the linear FE discretization of the Poisson equation and discuss, in the context of DMPs, the representation of the stiffness matrix entries in terms of the edge lengths and cotangents of the dihedral angles of the tetrahedral meshes, thus avoiding the common use of face areas and element volumes.
Using this type of entry representation, certain structural properties of meshes, and simple trigonometric relations, as proposed in the two-dimensional case by Strang and Fix in \cite[p.~78]{Strang-Fix}, we show that DMPs still hold (as the stiffness matrices remain $M$-matrices)
for many tetrahedral configurations with obtuse dihedral angles. Such a relaxation of the standard acuteness/nonobtuseness conditions can be vital in practical calculations, as these angle properties are difficult to preserve in global and local mesh refinements.
Several representative examples are presented.
 }

\section{Introduction}
\label{Sec-Introduction}

The {\it maximum principles} (MPs) are the basic qualitative properties of many elliptic-type equations and systems of such equations (see, for example, \cite{Ladyzhenskaya-Uraltseva,Protter-Weinberger}), which adequately imitate the physical behavior of the modelled phenomena. It is natural to require that the numerical approximations for such models satisfy suitable discrete analogs, commonly called the {\it discrete maximum principles} (DMPs) of MPs.

To the best of our knowledge, the first (rather simple) variant of the DMP (for the finite difference method (FDM)) was presented by Varga in 1966  (see  \cite{Varga-1966}). Later, in the 70th, more sophisticated versions of DMPs, suitable for both - FEM and FDM schemes,
were developed by Ciarlet and Raviart in \cite{Ciarlet-1970} and \cite{Ciarlet-Raviart}. In \cite{Ciarlet-Raviart}, for the first time, the angle conditions of acuteness and nonobtuseness guaranteeing DMPs were proposed for underlying triangulations. In 1995, K\v r\'{\i}\v zek and Lin Qun \cite{Krizek-Qun} analyzed the validity of the DMP for nonlinear elliptic problems, also considering the effect of quadrature rules. They proposed the notion of nonobtuse tetrahedral meshes in which all dihedral angles of the tetrahedral elements are nonobtuse. This angle condition guarantees  DMPs for many elliptic-type PDE models in the three-dimensional case.
The concepts of acuteness/nonobtuseness (in terms of dihedral angles between facets) are also applicable for proving DMPs
in the higher-dimensional case (of simplicial meshes), see e.g. \cite{Brandts-Korotov-Krizek-2007, Brandts-Korotov-Krizek-2008, Xu-Zikatanov-1999}, and for various nonlinear elliptic-type models
\cite{Karatson-Korotov-2005, Karatson-Korotov-Krizek-2007}. The generation and refinements of such simplicial meshes are reviewed in \cite{Brandts-Korotov-Krizek-Solc-2009}.

All the proofs in the above-mentioned works (see also references therein) are based on the observation that the produced FEM matrices are $M$-matrices (monotone matrices with non-positive off-diagonal entries) if the simplicial meshes are nonobtuse (or acute if the reactive-type terms are presented in the equations). In more detail, standard calculations of the off-diagonal entries of FEM matrices in the case of linear simplicial finite elements involve a summation of  terms like
\begin{equation}
\int\nolimits_{T} \nabla \phi_i \cdot \nabla \phi_j \, dx
=
- \frac{|F_i|\,|F_j|}{d^2\,|T|}\cos(\alpha_{ij}),
\label{eq:product-basis-functions}
\end{equation}
where $d$ is the dimension, $\phi_i$ and $\phi_j$ are the (linear) basis functions associated with the vertices of a simplex $T$, labelled by the indices $i$ and $j$,  $|F_i|$ and $|F_j|$ are the areas of the facets of $T$ opposite to those vertices, $|T|$ is the volume of $T$, and $\alpha_{ij}$ is the dihedral angle of $T$ formed by the facets $F_i$ and $F_j$. From this representation, it is evident that such terms are always nonpositive if all the dihedral angles are nonobtuse.
We notice that in the computations based on representation \eqref{eq:product-basis-functions}, we need the areas of the facets and volumes of the mesh elements.

However, it is widely known that (see e.g. \cite{Shewchuk-2002, Strang-Fix}) in two dimensions, the off-diagonal entries in the FEM matrix, e.g. appearing in the FEM discretization of the Poisson equation, are computable in terms of cotangents of angles as follows
\begin{equation}
    a_{ij} :=
    \int\nolimits_{\Omega} \nabla \phi_j \cdot \nabla \phi_i \, dx
    = -\frac{1}{2}\left(\cot\alpha + \cot\beta\right) \quad (i \neq j),
    \label{eq:off-diagonal-entry-2D}
\end{equation}
where $\alpha$ and $\beta$ are the angles in two adjacent triangles $T_1$ and $T_2$ opposite to their common  edge connecting vertices with indices
$i$ and $j$ ($B_i$ and $B_j$ in Figure~\ref{exm2a}). It is worth to note that  formula \eqref{eq:off-diagonal-entry-2D} can be immediately derived from \eqref{eq:product-basis-functions},
using the well-known formula for the area of a triangle.

\begin{figure}[h!]
    \centering
    \includegraphics[width=0.6\textwidth]{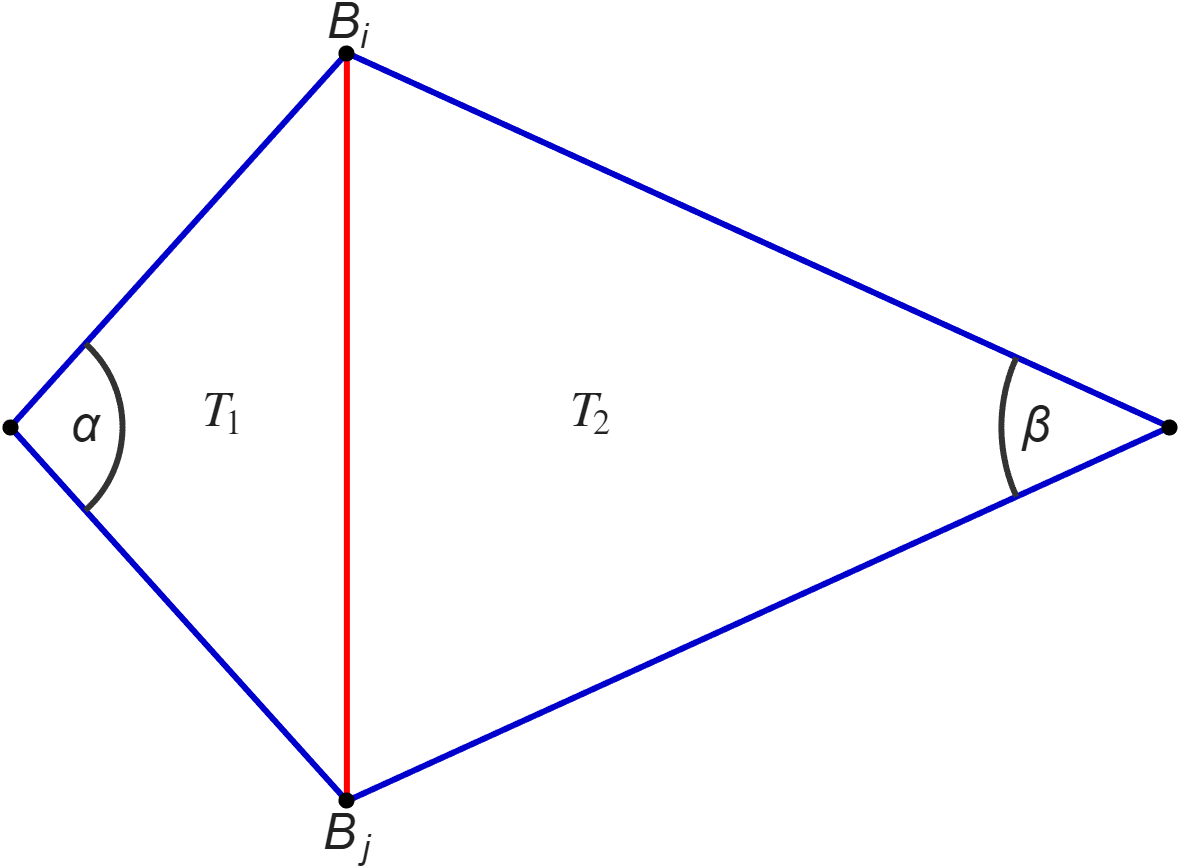}
    \caption{Illustration for formula \eqref{eq:off-diagonal-entry-2D}.}
    \label{exm2a}
\end{figure}

As remarked by Strang and Fix in their monograph, \cite[p.\,78]{Strang-Fix}, using \eqref{eq:off-diagonal-entry-2D}, the definition of cotangent, and a simple trigonometric equality for the sine of a sum of two angles, one immediately sees from     \eqref{eq:off-diagonal-entry-2D} that
\begin{equation}
    a_{ij} = -\frac{1}{2} \frac{\sin(\alpha + \beta)}{\sin \alpha \sin \beta}.
\end{equation}
Therefore, $a_{ij} \le 0$ (and, thus, DMP is valid, see the next section for relevant details) if and only if
\begin{equation}
 \alpha + \beta \le \pi,
 \label{eq:alpha-plus-beta}
\end{equation}
which, obviously, weakens the nonobtuseness requirement on each element in the triangulations, still staying in the framework of
$M$-matrices. We mention here  that condition \eqref{eq:alpha-plus-beta}
is closely related to the properties of Delaunay triangulations.

In this work, we discuss the presentation of $a_{ij} \ (i \neq j)$ in terms of edge lengths and cotangents of dihedral angles in dimensions higher than two, which is the generalizion of \eqref{eq:off-diagonal-entry-2D}, give a new proof for such a presentation in the case of tetrahedral meshes, and, in the spirit of the above-mentioned remark by Strang and Fix, show how such a representation can be used, in certain situations, for example, for some well-structured meshes, to prove DMPs for tetrahedral meshes containing obtuse dihedral angles.

\section{Model problem and maximum principles}
\label{Sec-Model-MPs}

For the denotations and  definitions of various functional spaces used in what follows, see e.g. \cite{Ciarlet-book}.

\medskip

We consider the following test problem: Find a function $u$ such that
\begin{align}
- \Delta u &= f \quad \text{in } \ \Omega, \label{probdef1} \\
u &= 0 \quad \text{on } \partial \Omega, \label{probdef2}
\end{align}
where $f$ is a sufficiently smooth function.
The domain $\Omega \subset \mathbb{R}^d$ is a bounded polytopic domain with
a Lipschitz continuous boundary $\partial \Omega$.

The classical solutions of elliptic problems of the second order are known to satisfy the MPs \cite{Ladyzhenskaya-Uraltseva,Karatson-Korotov-2005}, which, in our case, means the following implication:
\begin{equation}
    f \le 0 \quad \Rightarrow \quad \max\limits_{x \in \overline \Omega} \, u(x) \le 0.
    \label{eq:MP}
\end{equation}

\begin{remark}
 Sometimes, \eqref{eq:MP} is called the {\it nonpositivity/nonnegativity preservation}.
 In case of nonzero Dirichlet boundary conditions it takes a slightly different form,
 see e.g. \cite{Karatson-Korotov-2005}.
\end{remark}

The standard FEM schemes for approximating the (unkown) solution $u$ of \eqref{probdef1}--\eqref{probdef2}
are based on the so-called {\it weak formulation}: Find $u \in H^1_0(\Omega)$ such that
$$
a(u,v) = \mathcal F(v) \quad \forall v \in H^1_0(\Omega),
$$
where
\begin{equation*}
    a(u,v) = \int_{\Omega} \nabla u \cdot \nabla v \, dx  \quad \text{and} \quad \mathcal F(v) = \int_{\Omega} f v \, dx,
\end{equation*}
where, in addition, we assume that $f \in L^2(\Omega)$.

\begin{remark}
 The existence and uniqueness of the weak solution $u$ are provided by the standard Lax-Milgram lemma.
\end{remark}

Let $\mathcal{T}_h$ be a simplicial face-to-face partition ({\it mesh}) of $\overline \Omega$ with nodes $$B_1, \dots, B_{N+N^{\partial}}$$
from which $B_1, \dots, B_N$ are the interior ones, and the others, $B_{N+1}, \dots , B_{N+N^{\partial}}$, lie on the boundary $\partial \Omega$.
We define continuous piecewise linear basis functions $\phi_i, \ i = 1, \dots, N+ N^{\partial}$, associated with the nodes of $\mathcal{T}_h$,
so that
\begin{equation*}
    \phi_i(B_j) = \delta_{ij},
\end{equation*}
where $\delta_{ij}$ is the Kronecker symbol.   Obviously, all the basis functions defined above are nonnegative.
We shall use the basis functions related to the interior nodes to span the finite-dimensional subspace $V_h^0$ of $H^1_0(\Omega)$.

The FEM approximation for our test problem is then defined as a function $u_h \in V_h^0$ such that
$$
a(u_h,v_h) = \mathcal F(v_h) \quad \forall v_h \in V_h^0,
$$
whose existence and uniqueness are also provided by the Lax-Milgram lemma.

Algorithmically,
$$
u_h = \sum\limits_{i=1}^{N} y_i \phi_i,
$$
where the coefficients $y_i$ are the entries of the solution ${\bf y} = [y_1, \dots, y_N]^T$ of
the following square system of  $N$ linear algebraic equations
\begin{equation}
    {\bf A} {\bf y} = {\bf F},
    \label{eq:main-system}
\end{equation}
where the entries of ${\bf A}$ (which is often called the {\it stiffness matrix}, also in this paper) are defined as
\begin{equation*}
    a_{ij} = a(\phi_j, \phi_i), \quad i,j = 1,\dots, N,
\end{equation*}
and the entries of ${\bf F}$ are
\begin{equation*}
f_i = \mathcal F (\phi_i), \ i = 1, \dots, N.
\end{equation*}

The DMP for our model problem,  naturally imitating \eqref{eq:MP}, is the following statement:
\begin{equation}
    f \le 0 \quad \Rightarrow \quad \max\limits_{x \in \overline \Omega} \, u_h(x) \le 0.
    \label{eq:DMP}
\end{equation}

From the positivity of basis functions we see that $f \le 0$ implies that all entries $f_i$ of ${\bf F}$
are nonpositive and, in order to get the nonpositive FE approximation $u_h$, it is sufficient to guaranty
that all entries of ${\bf y}$ in the system \eqref{eq:main-system} are nonpositive. This property, in turn, can be obtained if the matrix ${\bf A}$ is monotone (invertible with a nonnegative inverse ${\bf A}^{-1}$).

One class of monotone matrices is formed by the so-called {\it Stieltjes matrices}, that is, matrices that are symmetric and positive definite with non-positive off-diagonal entries \cite{Varga-book}. This concept was used to prove DMPs for the class of elliptic-type problems, including our test problem, for example, in \cite{Hannukainen-Korotov-Vejchodsky-2008, Brandts-Korotov-Krizek-2008}.
As the stiffness matrices appearing in the above-described discretization are symmetric and positive definite, the only condition to guarantee in the DMP context is the correct sign of the off-diagonal entries. This will be the subject of the next section.

\begin{remark}
The cases of other (classical) boundary conditions (non-zero Dirichlet and mixed ones) can be treated similarly
to the zero Dirichlet boundary condition, with small modifications of technical character; see, for example, \cite{Hannukainen-Korotov-Vejchodsky-2008,  Karatson-Korotov-2005} for details.
\end{remark}

\section{On computation of the off-diagonal entries}
Let $T$ be a tetrahedron with vertices $B_1,B_2,B_3,B_4$.
For each pair of distinct indices $i \neq j$, the edge connecting
$B_i$ and $B_j$ and its length are denoted by $\ell_{ij}$ and
$|\ell_{ij}|$. The dihedral angle $\alpha_{ij}$ adjacent to $\ell_{ij}$ is
formed by two triangular faces $S_{ijk}$ and $S_{ijl}$ of $T$, where we use the following general agreement
on indices,
$$
\{k,l\}=\{1,2,3,4\}\setminus\{i,j\}.
$$
The above introduced denotations are illustrated in Figure~\ref{tetrahedron-T}. We notice that, for the case of tetrahedra,
we use a slightly different denotation for the faces
than those used in \eqref{eq:product-basis-functions} for general simplices.

\begin{figure}[H]
    \centering
    \includegraphics[width=0.4\textwidth]{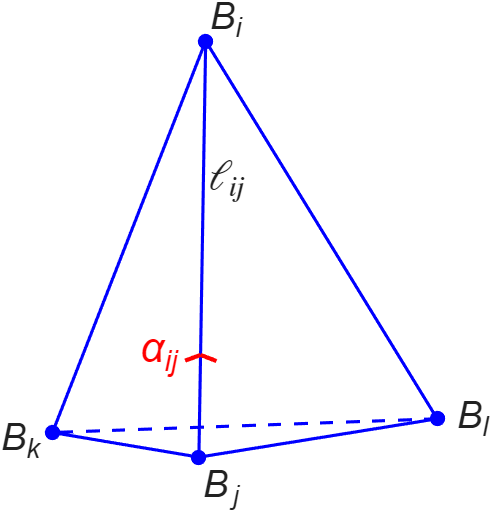}
    \caption{Denotations related to the tetrahedron $T$.}
    \label{tetrahedron-T}
\end{figure}

The {\it Cayley--Menger matrix} associated with the tetrahedron $T$ is
\begin{equation}
C(T) =
\begin{pmatrix}
0 & |\ell_{12}|^2 & |\ell_{13}|^2 & |\ell_{14}|^2 & 1 \\
|\ell_{21}|^2 & 0 & |\ell_{23}|^2 & |\ell_{24}|^2 & 1 \\
|\ell_{31}|^2 & |\ell_{32}|^2 & 0 & |\ell_{34}|^2 & 1 \\
|\ell_{41}|^2 & |\ell_{42}|^2 & |\ell_{43}|^2 & 0 & 1 \\
1 & 1 & 1 & 1 & 0
\end{pmatrix},
\end{equation}
and its determinant $\mathcal D $ (called the {\it Cayley--Menger determinant}),
is known to satisfy the relation
\begin{equation}
\mathcal D  = 288\,|T|^2.
\label{CM_volume}
\end{equation}

We further associate with the Cayley--Menger matrix a collection of cofactors with direct geometric significance. Let $C_{[l]}(T)$ be the $4 \times 4$ submatrix of $C(T)$ obtained by deleting the row and column indexed by $l$. Let
\begin{equation}
\mathcal{D}_{ijk} := \det C_{[l]}(T),
\end{equation}
\begin{equation}
\mathcal{D}_{ij} := (-1)^{k+l} \det C_{[k,l]}(T),
\end{equation}
where $C_{[k,l]}(T)$ denotes the $3 \times 3$ submatrix of $C(T)$ obtained by deleting the rows and columns indexed by $k$ and $l$.

\medskip

The following two results were proved in \cite{Karl-Wirth-2014}.

\medskip

\begin{lemma}
Let $\{i,j,k,l\} = \{1,2,3,4\}$, the following identities hold
\begin{equation}
\mathcal{D}_{ijk} = -16|S_{ijk}|^2,
\label{area_Id}
\end{equation}
\begin{equation}
\mathcal{D}_{ijk}\mathcal{D}_{ijl}
= \mathcal{D}_{ij}^2 + 2|\ell_{ij}|^2\mathcal{D},
\label{Cofactor_Id}
\end{equation}
where $|S_{ijk}|$ stands for the area of the face $S_{ijk}$.

\end{lemma}
\medskip
\begin{theorem} \label{theorem_cofactor}
The dihedral angle $\alpha_{ij}$ associated with the edge $\ell_{ij}$ of a nondegenerate tetrahedron $T$ satisfies
\begin{equation}
\cos(\alpha_{ij})
=
\frac{\mathcal{D}_{ij}}{\sqrt{\mathcal{D}_{ijl}\,\mathcal{D}_{ijk}}}.
\end{equation}
\end{theorem}

Now, using the above-presented notations and relations, we can prove the main results of this section.

\begin{theorem} \label{stiff:cot}
Let $T$ be a nondegenerate tetrahedral finite element with linear basis functions $\phi_1,\phi_2,\phi_3,\phi_4$.
Then for each pair of distinct indices $i \neq j$, the relevant contribution from $T$ to the off-diagonal stiffness entry $a_{ij}$ satisfies
\begin{equation}
a_{ij}|_T
:=
\int_T \nabla \phi_j \cdot \nabla \phi_i \, dx
=
-\frac{1}{6} |\ell_{k l}| \cot(\alpha_{kl}).
\end{equation}
\end{theorem}
\begin{proof}
 From the area identity \eqref{area_Id}, it follows that
\begin{equation*}
    \mathcal{D}_{ijl}\,\mathcal{D}_{ijk} = 256\,|S_{ijl}|^2\,|S_{ijk}|^2.
\end{equation*}
Using Theorem~\ref{theorem_cofactor} for the  identity $\sin(\alpha_{ij}) = \sqrt{1 - \cos^2(\alpha_{ij})}$,
we observe that
\begin{align*}
\sin(\alpha_{ij})
=
\sqrt{1 - \frac{\mathcal{D}_{ij}^2}
{\mathcal{D}_{ijl}\mathcal{D}_{ijk}}}
=
\sqrt{\frac{2|\ell_{ij}|^2\mathcal{D}}
{\mathcal{D}_{ijl}\mathcal{D}_{ijk}}}
=
\frac{|\ell_{ij}|\sqrt{2\mathcal{D}}}
{\sqrt{\mathcal{D}_{ijl}\mathcal{D}_{ijk}}}
=
\frac{24|\ell_{ij}||T|}
{16|S_{ijl}||S_{ijk}|}.
\end{align*}
Solving for $|T|$ gives
\begin{equation}\label{volume_identity}
|T| = \frac{2}{3}\,\frac{|S_{ijk}|\,|S_{ijl}|}{|\ell_{ij}|}\,\sin(\alpha_{ij}).
\end{equation}

Because the indices used in the above derivation are arbitrary, we can also write it as follows:

\begin{equation}\label{volume_identity-2}
|T| = \frac{2}{3}\,\frac{|S_{kli}|\,|S_{klj}|}{|\ell_{kl}|}\,\sin(\alpha_{kl}).
\end{equation}

We know that, see e.g. \cite{Krizek-Qun},
\begin{equation}
    a_{ij}|_T
=
-\frac{|S_{kli}|\,|S_{klj}|}{9\,|T|}\cos(\alpha_{kl}).
\end{equation}
Substituting \eqref{volume_identity-2} into the above expression, we obtain
\begin{align}
a_{ij}|_T
&=
-\frac{|S_{kli}|\,|S_{klj}|}{9}
\cdot
\frac{3|\ell_{kl}|}
{2|S_{kli}||S_{klj}|\sin(\alpha_{kl})}
\cos(\alpha_{kl})
\\
&=
-\frac{|\ell_{kl}|}{6}
\frac{\cos(\alpha_{kl})}
{\sin(\alpha_{kl})} = -\frac{|\ell_{kl}|}{6}
\cot({\alpha_{kl}}).
\end{align}

\hfill $\square$
\end{proof}

\begin{corollary}
Let $a_{ij}$ denote the global stiffness matrix entry associated with the
edge $\ell_{ij}$. Then
\begin{equation}
    a_{ij}
=
-\frac{1}{6}
\sum_{T \in \mathcal{P}_{ij}}
|\ell_{kl}^T|
\cot\!\bigl(\alpha_{kl}^{T}\bigr),
\label{eq:a-ij-full}
\end{equation}
where $\mathcal{P}_{ij}$ denotes the cluster (or patch) of tetrahedra sharing the edge
$\ell_{ij}$, and $\alpha_{kl}^{T}$ stands for the dihedral angle corresponding to the edge
$\ell_{kl}^T$ (opposite to $\ell_{ij}$) in the tetrahedron $T$ from the cluster/patch $\mathcal{P}_{ij}$.
\end{corollary}
\begin{proof}
The global stiffness matrix entry $a_{ij}$ is obtained by summing the
corresponding local contributions over all tetrahedra
$T \in \mathcal{P}_{ij}$. Therefore, the result follows directly from Theorem \ref{stiff:cot}.
\hfill $\square$
\end{proof}

\begin{remark}
Different proofs of \eqref{eq:a-ij-full} can be found in \cite{Barth-1992, Xu-Zikatanov-1999}.
\end{remark}

\begin{remark}
The representation of off-diagonal entries in form \eqref{eq:a-ij-full}
gives certain implementation benefits in comparison with representation \eqref{eq:product-basis-functions}.
For example, we need not compute face areas and volumes of the mesh elements.
\end{remark}

If the representation \eqref{eq:a-ij-full} is used then there is a possibility, in the spirit of Strang-Fix
passage, to compensate positive contributions to $a_{ij}$ (appearing due to obtuse dihedral angles) by
certain negative contributions if the corresponding edge lengths are, for example, equal.
Two examples of tetrahedral clusters employing this idea are given in the last section.
\begin{remark}
In general, a possibility to allow (slightly) obtuse dihedral angles in the tetrahedral FE meshes
is quite attractive from a practical point of view as the nonobtuseness property is hard to
preserve during various global and especially local refinements of simplicial FE meshes which we perform in
order to compute more accurate FE approximations, see e.g. \cite{Brandts-Korotov-Krizek-Solc-2009}
for a relevant discussion.
\end{remark}

\section{Examples and final comments}
\label{sec4}

In this section, we present two examples of tetrahedral clusters, with some tetrahedra having obtuse dihedral angles, which produce nonpositive off-diagonal entries in the corresponding stiffness matrices.

\begin{example} \label{ex-edge4tetra} See Figure~\ref{fig:tetra-configurations_1} for denotations. We compute the
entry $a_{12}$, assuming that the edge $B_1 B_2$ is surrounded by four tetrahedra and that the vertices
$B_3, B_4, B_5, B_6$ lie in the same plane and form a rectangle, i.e. $|\ell_{34}| = |\ell_{56}|$ and $|\ell_{36}| = |\ell_{45}|$. Let the vertices $B_1$ and $B_2$ be mirror images of each other with respect to the rectangle
$B_3 B_4 B_5 B_6$.
Further, we assume that both $B_1$ and $B_2$ lie (arbitrarily) on the circle $\mathcal O$, perpendicular to $B_3 B_4 B_5 B_6$, so that its center coincides with the center of $B_3 B_4 B_5 B_6$ and that $\mathcal O$ passes through the midpoints of the edges $\ell_{34}$ and $\ell_{56}$.
Then, it is clear that the dihedral angles (in the corresponding tetrahedra) adjacent to edges $\ell_{34}$ and $\ell_{56}$ sum to $\pi$. Thus, the sum of the two corresponding contributions to $a_{12}$ is zero, that is, nonpositive.
If we move (simultaneously) both points $B_1$ and $B_2$ along the upper and lower arcs of $\mathcal O$, respectively, towards to the midpoint of $\ell_{34}$ (or that of $\ell_{56}$ like in Figure~\ref{fig:tetra-configurations_1} (right)), one can guarantee that the dihedral angles
(in relevant tetrahedra of the cluster) associated with the edges $\ell_{36}$ and $\ell_{45}$ become nonobtuse (even tending to zero) at some point, that is, the corresponding two contributions to $a_{12}$ are nonpositive.
As the result, the total summation in $a_{ij}$ gives the desired negative sign.
\end{example}

\smallskip

It is worth to notice that in the above example, we can produce an obtuse dihedral angle in the tetrahedron
$B_1 B_2 B_5 B_6$ arbitrarily close to $\pi$, with still having $a_{12} \le 0$.

\begin{figure}[h!]
    \centering
    \begin{minipage}{0.4\textwidth}
        \centering
        \includegraphics[width=\textwidth]{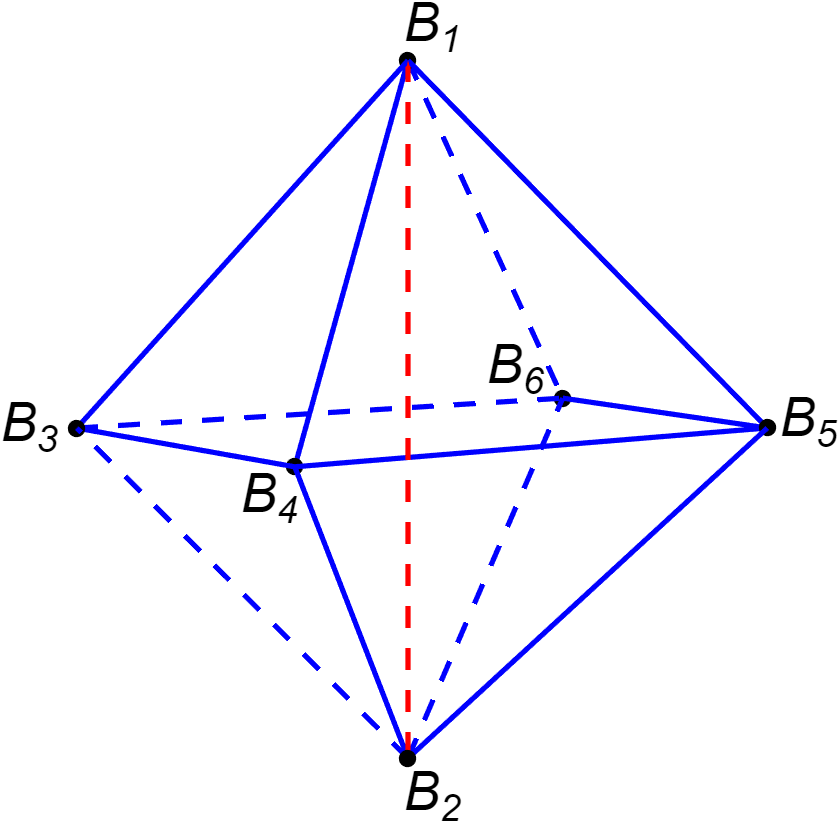}
    \end{minipage}
    \hfill
    \begin{minipage}{0.45\textwidth}
        \centering
              \includegraphics[width=\textwidth]{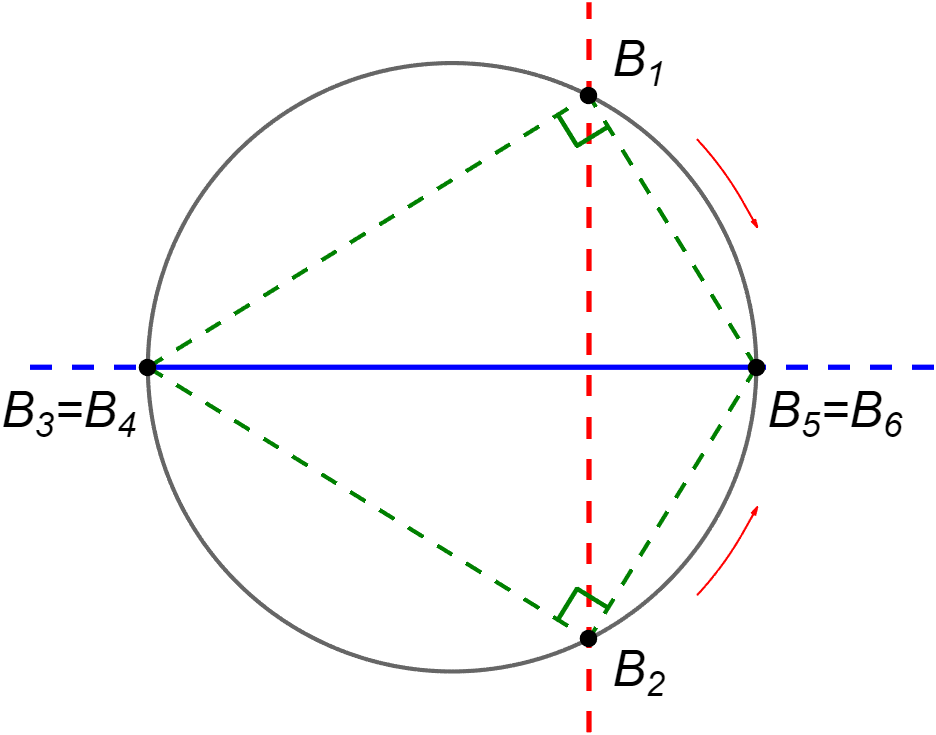}
    \end{minipage}
    \caption{Tetrahedral cluster of 4 tetrahedral elements in Example~\ref{ex-edge4tetra} (left). Cross-section of the cluster
    with $B_1$ and $B_2$ moving along the arcs of $\mathcal O$ toward to the midpoint of $B_5 B_6$ (right).}
    \label{fig:tetra-configurations_1}
\end{figure}

\begin{example} \label{ex-edge4tetraonrectanglebase} In this example, we assume that the vertex $B_2$ lies inside  the rectangle formed by the vertices $B_3, B_4, B_5, B_6$. The vertex $B_1$ is always in such a position that it is above the plane going through $B_3, B_4, B_5, B_6$, and its projection, denoted by $B_1'$, onto this plane lies outside the rectangle  $B_3 B_4 B_6 B_5$, and the straight line passing through $B_2$ and the projection $B_1'$ is parallel to the edges of the rectangle. For example, the situation can be illustrated as shown in Figure~\ref{fig:pyramid-configurations_2}. Then the dihedral angle adjacent to the edge $\ell_{56}$ is obtuse, but the sum of dihedral angles along edge $\ell_{56}$ and edge $\ell_{34}$  is less than $\pi$. See the triangle $P_1 B_1 P_2$, where
the points $P_1$ and $P_2$ are the crossections of the straight line through $B_2$ and $B_1'$ with two sides of the rectangle $B_3 B_4 B_5 B_6$. The other two dihedral angles along edge $\ell_{46}$ and edge $\ell_{35}$   are always acute. Thus, the
entry $a_{12}$ is even strictly negative for such a cluster.
\end{example}

\begin{figure}[h!]
    \centering
    \begin{minipage}{0.5\textwidth}
        \centering
        \includegraphics[width=\textwidth]{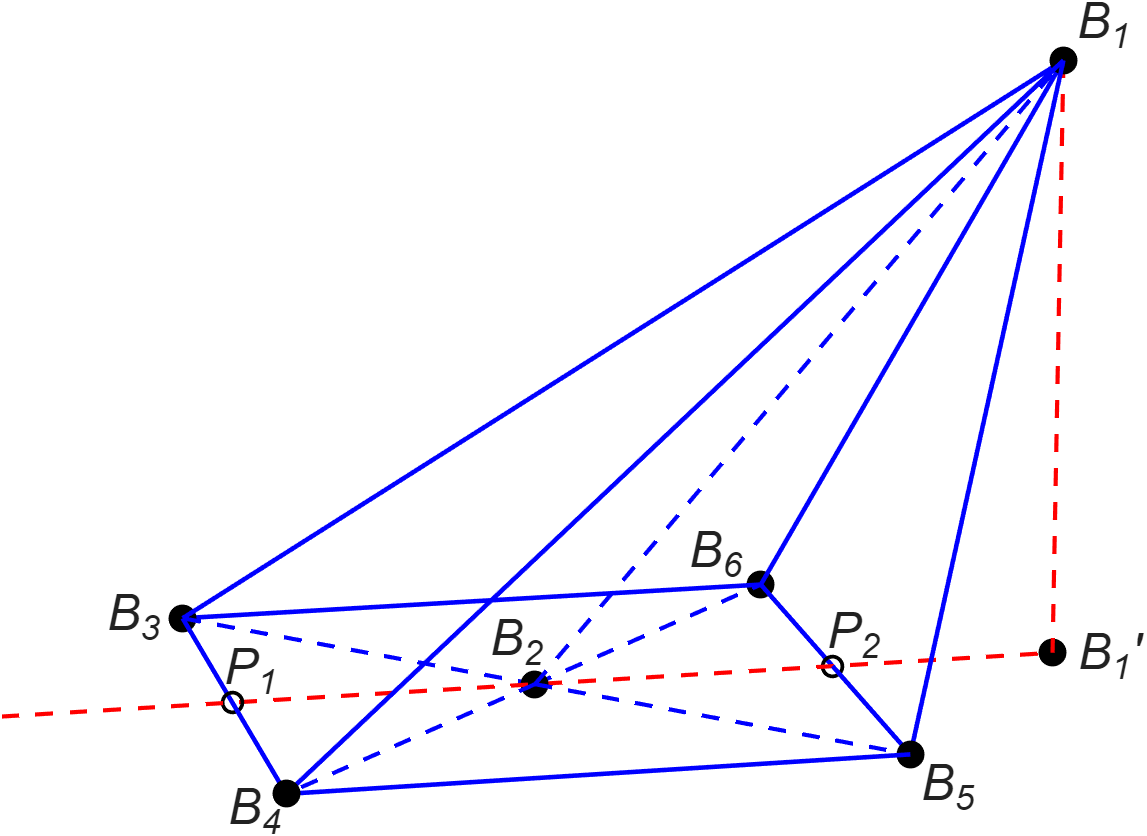}
    \end{minipage}
    \hfill
    \begin{minipage}{0.33\textwidth}
        \centering
              \includegraphics[width=\textwidth]{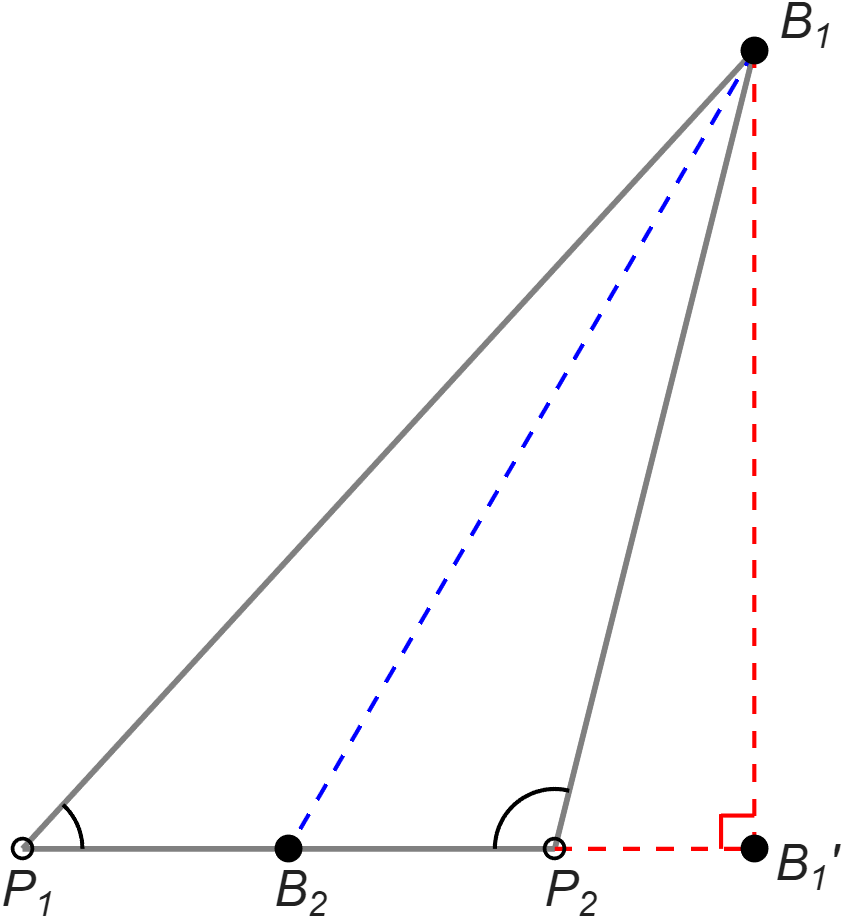}
    \end{minipage}
    \caption{Cluster of an oblique quadrilateral pyramid in Example~\ref{ex-edge4tetraonrectanglebase} (left). Cross-section through the auxiliary points $P_1$, $P_2$, and the vertex $B_1$, and the projection point $B_1'$ marked (right).}
    \label{fig:pyramid-configurations_2}
\end{figure}

As a conclusion, we can say that if we have a mesh, where each interior edge has a tetrahedral cluster around itself
in the manner of the above examples, then we have the validity of the DMP on that mesh (with some dihedral angles being obtuse) for our model problem. It seems that
such an approach can be succesfully used if the meshes in focus obey suitable structure properties.
This will be the subject of the forthcoming paper.



\begin{thebibliography}{99}
\bibitem{Barth-1992}
Barth, T.:
Aspects of unstructured grids and finite-volume solvers for the Euler and Navier-Stokes equations.
Tech. Report AGARD Report 787, AGARD, 1992

\bibitem{Brandts-Korotov-Krizek-2007}
Brandts, J., Korotov, S.,  K\v{r}\'{i}\v{z}ek, M.: Dissection of the path-simplex in $R^n$ into $n$ path-subsimplices. Linear Algebra Appl. \textbf{421}(2-3), 382-393 (2007)

\bibitem{Brandts-Korotov-Krizek-2008}
Brandts, J. H., Korotov, S., K\v{r}\'{i}\v{z}ek, M.:
The discrete maximum principle for linear simplicial finite element approximations of a reaction--diffusion problem.
Linear Algebra Appl. \textbf{429}(10), 2344--2357 (2008)

\bibitem{Brandts-Korotov-Krizek-Solc-2009}
Brandts, J., Korotov, S., K\v{r}\'{i}\v{z}ek, M., \v{S}olc, J.:
On nonobtuse simplicial partitions.
SIAM Rev. \textbf{51}(2), 317--335 (2009)

\bibitem{Ciarlet-1970} Ciarlet, P. G.: Discrete maximum principle for finite-difference operators. Aequationes Math. \textbf{4}(3), 338-352 (1970)


\bibitem{Ciarlet-book}
Ciarlet, P. G.: The Finite Element Method for Elliptic Problems. North-Holland, Amsterdam (1978)



\bibitem{Ciarlet-Raviart} Ciarlet, P. G., Raviart, P. A.:  Maximum principle and uniform convergence for the finite element method. Comput. Methods Appl. Mech. Eng. \textbf{2}(1), 17-31 (1973)

\bibitem{Hannukainen-Korotov-Vejchodsky-2008}
Hannukainen, A., Korotov, S., Vejchodsk\'{y}, T.:
On weakening conditions for discrete maximum principles for linear finite element schemes.
In: Margenov, S., Vulkov, L.G., Waśniewski, J. (eds) Numerical Analysis and Its Applications. NAA 2008. Lecture Notes in Computer Science, vol 5434. Springer, Berlin, Heidelberg,  297--304

\bibitem{Karatson-Korotov-2005}
Kar\'{a}tson, J., Korotov, S.:
Discrete maximum principles for finite element solutions of nonlinear elliptic problems with mixed boundary conditions.
Numer. Math. \textbf{99}(4), 669--698 (2005)

\bibitem{Karatson-Korotov-Krizek-2007}
 Kar\'{a}tson, J., Korotov, S., K\v{r}\'{i}\v{z}ek, M. On discrete maximum principles for nonlinear elliptic problems. Math. Comput. Simul. \textbf{76}(1-3), 99--108 (2007)


\bibitem{Krizek-Qun} K\v{r}\'{i}\v{z}ek, M.,  Qun, L.: On diagonal dominance of stiffness matrices in 3D. East-West J. Numer. Math, \textbf{3}(1), 59-69 (1995)


\bibitem{Ladyzhenskaya-Uraltseva}
Ladyzhenskaya, O. A., Ural'tseva, N. N. Linear and Quasilinear Elliptic Equations. Academic Press, New York, 1968

\bibitem{Protter-Weinberger} Protter, M. H., Weinberger, H. F.: Maximum Principles in Differential Equations. Springer-Verlag, New York, 1984


\bibitem{Shewchuk-2002} Shewchuk, J. R.: What is a good linear element? Interpolation, conditioning, and quality measures. In: N. Chrisochoides (eds) Proceedings of the 11th International Meshing Roundtable, IMR 2002, 115-126 (2002)

\bibitem{Strang-Fix} Strang, G., Fix, G. J.: An Analysis of the Finite Element Method. Prentice-Hall, Englewood Cliffs, NJ, 1973

\bibitem{Varga-1966} Varga, R. S.: On a discrete maximum principle. SIAM J. Numer. Anal. \textbf{3}(2), 355--359 (1966)

\bibitem{Varga-book} Varga, R. S.: Matrix iterative analysis. Springer Science and Business Media, Vol. 27, 1999

\bibitem{Xu-Zikatanov-1999}
Xu, J., Zikatanov, L.:
A monotone finite element scheme for convection-diffusion equations.
Math. Comput. \textbf{68}(228), 1429--1446 (1999)

\bibitem{Karl-Wirth-2014} Wirth, K., Dreiding, A. S.:
Relations between edge lengths, dihedral and solid angles in tetrahedra.
J. Math. Chem. \textbf{52}(6), 1624--1638 (2014)

\end{thebibliography}
\end{document}